\documentclass[12pt]{article}

\usepackage{color}
\usepackage{graphics,amsmath,amssymb}
\usepackage{amsthm}
\usepackage{amsfonts}
\usepackage{latexsym}
\usepackage{url}
\usepackage{fullpage}
\usepackage{float}
\usepackage{epsf}
\usepackage[linkcolor=webgreen,filecolor=webbrown,citecolor=webgreen]{hyperref}

\graphicspath{{figures/}}

\newtheorem{theorem}{Theorem}
\newtheorem{definition}{Definition}

\newcommand{\vcentergraphics}[1]{\ensuremath{\vcenter{\hbox{\includegraphics{#1}}}}}
\newcommand{\st}[2]{\ensuremath{\vcenter{\hbox{\scalebox{.7}{\includegraphics{ternary_tree_#1-#2}}}}}}
\newcommand{\seqnum}[1]{\href{http://oeis.org/#1}{\underline{#1}}}

\begin{document}

\begin{center}
\vskip 1cm{\LARGE\bf Pattern Avoidance in Ternary Trees
}
\vskip 1cm
\large
Nathan Gabriel\footnote{Partially supported by NSF grant DMS-0851721}\\
Department of Mathematics\\
Rice University\\
Houston, TX 77251, USA \\
\ \\
Katherine Peske\footnotemark[1]\\ 
Department of Mathematics and Computer Science\\
Concordia College\\
Moorhead, MN 56562, USA\\
\ \\
Lara Pudwell\footnotemark[1]  \\
Department of Mathematics and Computer Science\\
Valparaiso University\\
Valparaiso, IN 46383, USA\\
\href{mailto:Lara.Pudwell@valpo.edu}{\tt Lara.Pudwell@valpo.edu}\\
\ \\
Samuel Tay\footnotemark[1]\\
Department of Mathematics \\
Kenyon College\\
Gambier, OH 43022, USA
\end{center}

\begin{abstract}
This paper considers the enumeration of ternary trees (i.e., rooted ordered trees in which each vertex has 0 or 3 children) avoiding a contiguous ternary tree pattern.  We begin by finding recurrence relations for several simple tree patterns; then, for more complex trees, we compute generating functions by extending a known algorithm for pattern-avoiding binary trees.  Next, we present an alternate one-dimensional notation for trees which we use to find bijections that explain why certain pairs of tree patterns yield the same avoidance generating function.  Finally, we compare our bijections to known ``replacement rules'' for binary trees and generalize these bijections to a larger class of trees.
\end{abstract}

\section{Introduction}\label{S:Intro}

The notion of one object avoiding another has been studied in permutations, words, partitions, and graphs.  Although pattern avoidance has proven to be a useful language to describe connections between various combinatorial objects, it has also attracted broad interest as a pure enumerative topic.  One combinatorial problem that has received much attention in recent years is to count the number of permutations of length $n$ avoiding a certain smaller permutation.  Here, permutation $\pi$ avoiding permutation $\rho$ means that $\pi$ has no subsequence that is order-isomorphic to $\rho$.  Although the classical case of the permutation pattern problem allows $\rho$ to be given as \emph{any} subsequence of $\pi$, a special case that can be attacked successfully via a variety of techniques is studying when $\pi$ contains $\rho$ as a consecutive subpermutation.  This latter question can be answered by a variety of techniques including inclusion-exclusion.  There also exist algorithmic techniques, such as the Goulden-Jackson cluster method \cite{GJ79, NZ99} to approach this question using generating functions.  Two natural questions arise: ``Given a permutation $\rho$, how many permutations of length $n$ avoid $\rho$?''\ and ``When do two forbidden permutations $\rho_1$ and $\rho_2$ have the same avoidance generating function?''

In this paper we consider the analogous questions for plane trees. All trees in the paper are rooted and ordered.  We will focus on ternary trees, that is, trees in which each vertex has 0 or 3 (ordered) children.  A vertex with no children is a \emph{leaf} and a vertex with 3 children is an \emph{internal vertex}.  A ternary tree with $k$ internal vertices has $2k+1$ leaves, and the number of such trees is $\frac{\binom{3k}{k}}{2k+1}$ (OEIS \seqnum{A001764}).  It is clear then that there only exist ternary trees with an odd number of leaves.  The first few ternary trees are depicted in Figure \ref{F:trees}.

Conceptually, a plane tree $T$ avoids a tree pattern $t$ if there is no instance of $t$ anywhere inside $T$. Pattern avoidance in vertex-labeled trees has been studied in various contexts by Steyaert and Flajolet \cite{SF83}, Flajolet, Sipala, and Steyaert \cite{FSS90}, Flajolet and Sedgewick \cite{FS09}, and Dotsenko \cite{DTBA}.  Recently, Khoroshkin and Piontkovski \cite{KPTBA} considered generating functions for general unlabeled trees but in a different context.

In 2010, Rowland \cite{Rowland10} explored contiguous pattern avoidance in binary trees (that is, rooted ordered trees in which each vertex has 0 or 2 children).  He had two motivations for choosing these particular trees; first, there is a clear and natural definition of what it means for a rooted ordered tree to contain a contiguous ordered pattern that is unclear for general trees, and second, there is natural bijection between $n$-leaf binary trees and $n$-vertex trees.  His study had two main objectives.  First, he developed an algorithm to find the generating function for the number of $n$-leaf binary trees avoiding a given tree pattern; he adapted this to count the number of occurrences of the given pattern.  Second, he determined equivalence classes for binary tree patterns, classifying two trees $s$ and $t$ as equivalent if the same number of $n$-leaf binary trees avoid $s$ as avoid $t$ for $n \geq 1$.  He completed the classification for all binary trees with at most eight leaves, using these classes to develop replacement bijections between equivalent binary trees.

In this paper, we extend Rowland's work by exploring pattern avoidance in ternary trees, and in some cases to general $m$-ary trees (that is, trees where each vertex has 0 or $m$ children).  We first compute recurrence relations to count trees that avoid ternary tree patterns with at most seven leaves.  Next, we adapt Rowland's algorithm to find functional equations for the avoidance generating functions of arbitrary ternary tree patterns.  Finally, we give bijections between trees avoiding several pairs of equivalent tree patterns, and begin generalizing this process to fit the more general case of $m$-ary trees.  The appendix contains all the equivalence classes of ternary tree patterns with at most nine leaves found using the \texttt{Maple} package TERNARYTREES.  The \texttt{Maple} package itself is given at the third author's website (\url{http://faculty.valpo.edu/lpudwell/maple.html}).

\subsection{Definitions and Notation}\label{S:Defs}

Following Rowland's definition of avoidance, a ternary tree $T$ \emph{contains} $t$ as a tree pattern if $t$ is a contiguous, rooted, and ordered subtree of $T$.  Conversely, $T$ \emph{avoids} $t$ if there is no such subtree of $T$.  For example, consider the three trees shown in Figure \ref{Feg1}.  $T$ contains $t$ because this pattern occurs beginning at the center child of the root of $T$ (see bolded subtree).  However, $T$ avoids $s$ because no vertex in $T$ has children extending from both its left and center children.

\begin{figure}[htb]
\begin{center}
$T=$  \vcentergraphics{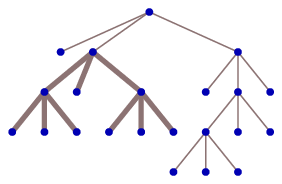} \hspace{.4in}
$t=$ \vcentergraphics{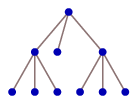} \hspace{.4in}
$s=$  \vcentergraphics{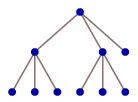}
\end{center}
\caption{Three ternary trees}
\label{Feg1}
\end{figure}

We define $\text{Av}_t(n)$ to be the set of $n$-leaf ternary trees that avoid the pattern $t$, and $\text{av}_t(n)= \left| \text{Av}_t(n) \right|$.  We will be particularly interested in determining the generating function $$\displaystyle{g_t(x)=\sum_{n=0}^\infty \text{av}_t(n) x^n}$$ for various patterns $t$.

Before we explore particular ternary tree patterns, we list all of the 3, 5, and 7-leaf ternary trees.  Note, however, if $t^r$ is the left--right reflection of $t$, then $\text{av}_t(n)=\text{av}_{t^r}(n)$ by symmetry, so left--right reflections are omitted.  We label trees with a double subscript notation.  The first subscript gives the number of leaves of the tree, and the second subscript distinguishes between distinct tree patterns of the same depth.  We will use these labels throughout the remainder of the paper.

\begin{figure}
\begin{center}
$t_{3_1}=$  \vcentergraphics{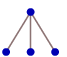} \hspace{.4in}
$t_{5_1}=$  \vcentergraphics{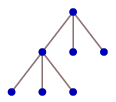} \hspace{.4in}
$t_{5_2}=$  \vcentergraphics{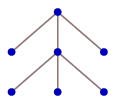} \\ \vspace{3mm}
$t_{7_1}=$  \vcentergraphics{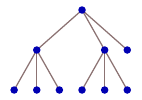} \hspace{.4in}
$t_{7_2}=$  \vcentergraphics{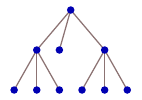} \hspace{.4in}
$t_{7_3}=$  \vcentergraphics{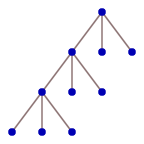} \hspace{.4in}
$t_{7_4}=$  \vcentergraphics{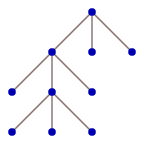} \\ \vspace{3mm}
$t_{7_5}=$  \vcentergraphics{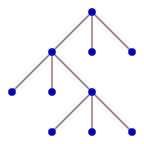} \hspace{.4in}
$t_{7_6}=$  \vcentergraphics{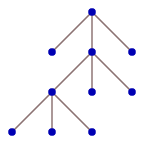} \hspace{.4in}
$t_{7_7}=$  \vcentergraphics{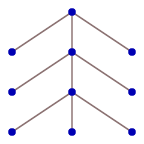}\\
\end{center}
\caption{3, 5, and 7-leaf ternary trees}
\label{F:trees}
\end{figure}

\section{Recurrences for Simple Tree Patterns}\label{S:recs}

In this section, we find recurrence relations for the number of trees avoiding several of the trees in Figure \ref{F:trees}.  For each tree, we discuss the structure of trees that avoid the given tree pattern, how a recurrence and generating function can be found from this structure, and we list any other equivalent tree patterns.  If $t$ is clear from context, we will simply write $\text{Av}(n)$ and $\text{av}(n)$ in lieu of $\text{Av}_t(n)$ and $\text{av}_t(n)$.

\subsection{Avoiding \texorpdfstring{$t_{5_1}$}{t51} and \texorpdfstring{$t_{5_2}$}{t52}}

To find $\text{av}_{t_{5_1}}(n)$, let us look at how an $n$-leaf tree $T$ must be structured in order to avoid $t_{5_1}$. Consider any internal vertex $v$ of $T$.  $v$'s left child can have no descendants, thus it must be a leaf. $v$'s center child can be the root of a subtree of any number of leaves $k$, where $1 \leq k \leq n-2$.  Finally, $v$'s right child can also be the root of a subtree, but because there are $n$ total leaves, this subtree must have precisely $n-k-1$ leaves.  Thus, there are $\text{av}(k)$ possible subtrees beginning at $v$'s center child, and $\text{av}(n-k-1)$ possible subtrees at $v$'s right child that also avoid $t_{5_1}$.  Taking the summation of these over the possible values of $k$ gives the recurrence relation $$\text{av}(n)=\sum_{k=1}^{n-2} \text{av}(k) \text{av}(n-k-1) \text{ where } n \geq 3.$$

Our initial conditions for this recurrence are $\text{av}(0) = 0$, because there are no ternary trees with 0 leaves; $\text{av}(1) = 1$, because there is one tree with one leaf, and it avoids any tree pattern with more than one leaf; and $\text{av}(2) = 0$, because there are no trees with 2 leaves.  We can now compute $g_{t_1}(x) = \sum_{k=0}^\infty \text{av}(k) x^k$ using standard techniques to obtain
$$g_{t_1}(x)=\frac{1 - \sqrt{1-4x^2}}{2x}.$$  The first few terms of this sequence are (for $n \geq 0$) $$0,1,0,1,0,2,0,5,0,14,\dots$$

Two things are worth noting about this avoidance sequence.  First, the non-zero terms are the Catalan numbers (OEIS \seqnum{A000108}).  Second, the sequence  is interspersed by zeros because there are no ternary trees with an even number of leaves.  This second observation will be true for the avoidance sequence of any ternary tree pattern.

For trees avoiding $t_{5_2}$, we only need to make one alteration; namely, that it is the center child, instead of the leftmost child, of each vertex that cannot have any children.  Therefore, we find that $$g_{t_{5_1}}(x)=g_{t_{5_2}}(x)=\frac{1 - \sqrt{1-4x^2}}{2x}.$$

\subsection{Avoiding \texorpdfstring{$t_{7_1}$}{t71} and \texorpdfstring{$t_{7_2}$}{t72}}

Next, we find the number of $n$-leaf trees that avoid $t_{7_1}$ such that $n\geq3$. As before, we consider any internal vertex $v$ of a tree $T$ that avoids $t_{7_1}$.  There are two nonexclusive possibilities for which of $v$'s children are internal vertices. First, $v$'s leftmost child has no children, but both its center and right children can. Otherwise, $v$'s center child has no children, but both its left and right children can.  These two cases are equivalent to avoiding $t_{5_1}$ and $t_{5_2}$, respectively.  However, this double-counts one instance: that is, when both the left and the center child of $v$ have no children. There are exactly $\text{av}(n-2)$ trees counted by both of the first two cases. Subtracting this from the recurrence relation, we are left with

$$\text{av}(n) = 2\sum_{k=1}^{n-2}\text{av}(k) \text{av}(n-k-1) - \text{av}(n-2).$$

Our initial conditions for this recurrence relation are again $\text{av}(0) = 0$, $\text{av}(1) = 1$, and $\text{av}(2) = 0$.  Using standard techniques, we obtain $$g_{t_{7_1}}(x) = \frac{x^2 +1 - \sqrt{x^4 - 6x^2 +1}}{4x},$$ which gives the Little Schr\"{o}der numbers (OEIS \seqnum{A001003}), interspersed by zeros:  0, 1, 0, 1, 0, 3, 0, 11, 0, 45, 0, 197, 0,$\dots$ 

This is also the avoidance sequence for $t_{7_2}$.  As before, two cases exist for avoiding $t_{7_2}$ (either the left and center or the right and center children of $v$ have descendants), as well as a term that needs to be subtracted to avoid double-counting (when neither the left nor the right children of $v$ have their own children).  Thus, we have $$g_{t_{7_1}}(x)=g_{t_{7_2}}(x) = \frac{x^2 +1 - \sqrt{x^4 - 6x^2 +1}}{4x}.$$

Before considering other tree patterns, we further examine the connection between trees avoiding $t_{7_2}$ and the Little Schr\"{o}der numbers ($s_n$).  To do this, we look at one well-known combinatorial interpretation of the Little Schr\"{o}der numbers: $s_n$ is the number of binary trees with $n$ vertices and with each right edge ``colored'' to be either solid or dashed \cite{Stanley99}.  We note that elsewhere in this paper, we have concerned ourselves with \emph{strictly} binary trees (each internal vertex has precisely 2 children) or \emph{strictly} ternary trees (each internal vertex has precisely 3 children).  In the current interpretation of the Little Schr\"{o}der numbers, however, these binary trees are not strict; that is, an internal vertex may have 1 or 2 children.  Consider the following map  from this set of colored (\emph{non-strict}) binary trees to the set of the (\emph{strict}) ternary trees avoiding $t_{7_2}$:

\begin{enumerate}
\item For each vertex $v$ in the binary tree, draw a vertex $v^*$.
\item Consider each parent-child pair in the binary tree:
\begin{itemize}
\item If $w$ is the left child of $v$, then  $w^*$ is the center child of $v^*$.
\item If $w$ is the right child of $v$ via a solid edge, then $w^*$ is the left child of $v^*$.
\item If $w$ is the right child of $v$ via a dashed edge, then $w^*$ is the right child of $v^*$.
\end{itemize}
\item For each vertex $v^*$ created in step 1, if $v^*$ has 0, 1, or 2 children, add children until $v^*$ has exactly 3 children.
\end{enumerate}

For example, the colored binary tree in Figure \ref{Fi:schroder} is mapped to the ternary tree in the same Figure.

\begin{figure}[bht]
\begin{center}
\vcentergraphics{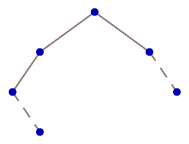} \hspace{.4in} $\rightarrow$ \hspace{.4in} \vcentergraphics{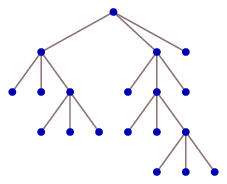}
\end{center}
\caption{A colored binary tree and its corresponding $t_{7_2}$-avoiding ternary tree}
\label{Fi:schroder}
\end{figure}

Note that any ternary tree that is produced by this algorithm certainly avoids $t_{7_2}$ since a vertex will not have both a solid right edge and a dashed right edge at the same time, and accordingly a vertex in the resulting ternary tree will never have both a right child and a left child with children of their own.  This process has an obvious inverse.

As an example, look at $s_3=11$.  There are eleven 3-vertex colored binary trees, and eleven 5-leaf trees avoiding $t_{7_2}$.  Each colored binary tree is shown with its image under our map in Figure \ref{F:sch}.

\begin{figure}[bht]
\begin{center}

\vcentergraphics{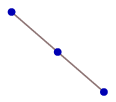} $\rightarrow$ \vcentergraphics{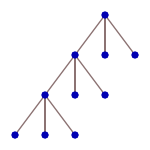} \hspace{.4in}
\vcentergraphics{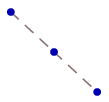} $\rightarrow$ \vcentergraphics{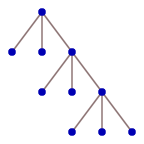} \hspace{.4in}
\vcentergraphics{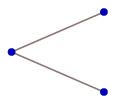} $\rightarrow$ \vcentergraphics{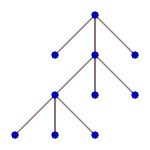} \\ \vspace{3mm}
\vcentergraphics{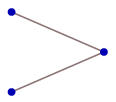} $\rightarrow$ \vcentergraphics{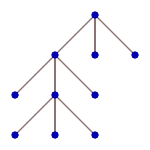} \hspace{.4in}
\vcentergraphics{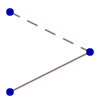} $\rightarrow$ \vcentergraphics{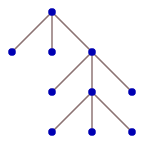} \hspace{.4in}
\vcentergraphics{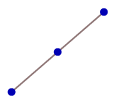} $\rightarrow$ \vcentergraphics{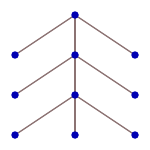} \\ \vspace{3mm}
\vcentergraphics{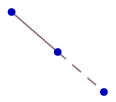} $\rightarrow$ \vcentergraphics{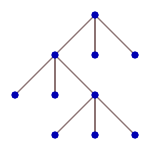} \hspace{.4in}
\vcentergraphics{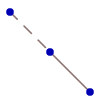} $\rightarrow$ \vcentergraphics{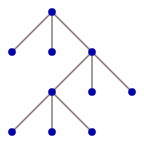} \hspace{.4in}
\vcentergraphics{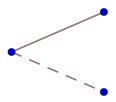} $\rightarrow$ \vcentergraphics{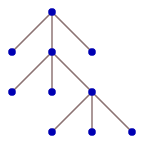} \\ \vspace{3mm}
\vcentergraphics{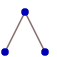} $\rightarrow$ \vcentergraphics{ternary_tree_7-1} \hspace{.4in}
\vcentergraphics{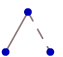} $\rightarrow$ \vcentergraphics{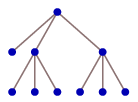} 
\end{center}
\caption{Mapping colored binary trees to $t_{7_2}$-avoiding ternary trees}
\label{F:sch}
\end{figure}

\subsection{Avoiding \texorpdfstring{$t_{7_3}$}{t73} and \texorpdfstring{$t_{7_7}$}{t77}}

To find the number of $n$-leaf trees that avoid $t_{7_3}$, $n\geq3$, we consider two cases for any internal vertex $v$ of a tree $T$ that avoids $t_{7_3}$. First, $v$'s left child has no children, while the center and right children are roots of subtrees with $k$ leaves and $n-k-1$ leaves respectively with $1 \leq k \leq n-2$. The second case is when $v$'s left child has three children; to avoid $t_{7_3}$, a left-vertex child cannot have another consecutive left-vertex child. The four other vertices (the center and right children of $v$ and the center and right children of $v$'s left child) are the roots of subtrees with $\ell$, $m$, $k$, and $n-\ell-m-k-1$ leaves respectively with $1 \leq \ell,m,k \leq n-4$. Therefore, $\text{av}_{t_{7_3}}(n)$ is given by the sum of these two cases
$$\text{av}(n)= \sum_{k=1}^{n-2} \text{av}(k) \text{av}(n-k-1) + \sum_{\ell=1}^{n-4} \sum_{m=1}^{n-\ell-3} \sum_{k=1}^{n-\ell-m-2} \text{av}(\ell) \text{av}(m) \text{av}(k) \text{av}(n-\ell-m-k-1)$$

To find the recurrence relation for trees avoiding $t_{7_7}$, we see that instead of avoiding two consecutive left-children vertices, we avoid two consecutive middle-children vertices. Therefore, $\text{av}_{t_{7_3}}(n)=\text{av}_{t_{7_7}}(n)$ for $n \geq 1$.  From this recurrence we compute the avoidance sequence: 0, 1, 0, 1, 0, 3, 0, 11, 0, 46, 0, 207, 0, $\dots$ (OEIS \seqnum{A006605} interspersed by zeros).

Clearly, it would be extremely difficult to solve this recurrence directly for the generating function $g_{t_{7_3}}(x)$.  It turns out that tree patterns $t_{7_4}$, $t_{7_5}$, and $t_{7_6}$, have the same avoidance generating function as $t_{7_1}$ and $t_{7_2}$, but we have not been able to find their recurrence relations by hand with an argument parallel to those above because complex problems arise with overcounting and undercounting. Instead, we now adapt Rowland's generating function algorithm for trees avoiding binary tree patterns to deal with ternary tree patterns.

\section{A Generating Function Algorithm}\label{S:gfAlgorithm}

As we saw in the previous section, it is straightforward to compute a given ternary tree pattern's avoidance generating function by hand for a few small tree patterns.  However, this type of computation quickly becomes impractical for increasingly complex tree patterns.  For this reason, we develop an algorithm to find a functional equation satisfied by the avoidance generating function $g_t(x)$ for any ternary tree pattern $t$.  First, however, we make one notational adjustment.  Now, let $g_{(t;p)}(x)$ be the generating function for the number of $n$-leaf ternary trees that avoid $t$ and contain the tree pattern $p$ at their root.  Recall, we have already specified that tree $T$ contains $t$ as a pattern if $T$ contains $t$ as a contiguous, rooted, ordered subgraph.  $T$ contains pattern $p$ at the root if $T$ contains a copy of $p$ where the root of $p$ coincides with the root of $T$.  Therefore, with our new notation, the generating function for all trees avoiding $t$ is given by $g_t(x)=g_{(t;\st{1}{1})}(x)$, because all ternary trees begin with the single vertex root.

The algorithm we use to find $g_{(t;\st{1}{1})}(x)$ is very similar to Rowland's algorithm for binary trees \cite{Rowland10}, but it accounts for an additional child at each internal vertex. The algorithm produces a sequence of generating functions using a recursive method. Initially, $g_{(t;\st{1}{1})}(x)$, the generating function we are interested in, is written in terms of another generating function. Then, for each new generating function $g_{(t;p)}(x)$ introduced in the recursive step, we deduce another recurrence in terms of other generating functions. If $t$ is a tree pattern, we will use $t_\ell$, $t_c$, and $t_r$ to denote the left, center, and right subtrees of $t$ respectively.  When appropriate we may write $t=\left(t_\ell t_ct_r\right)$.  We also require one more operation on tree patterns; we will use $s \cap t$ to denote the \emph{intersection} of tree patterns $s$ and $t$.  Conceptually, $s \cap t$ is the tree pattern produced by drawing $s$ and $t$ so that their roots coincide.  More formally, if $v$ is a single vertex, then $t \cap v = t$ and recursively $s \cap t = (s_\ell s_cs_r) \cap (t_\ell t_ct_r) = \left(\left(s_\ell \cap t_\ell \right)\left(s_c \cap t_c\right)\left(s_r \cap t_r\right)\right)$.  We note that although considering $s$ and $t$ as trees makes the intersection notation seem to be a misnomer, the set of trees with tree pattern $s$ at their root intersected with the set of trees with tree pattern $t$ at their root is indeed the set of trees with tree pattern $s \cap t$ at their root.  With this new notation, we are now prepared to give an algorithm to find $g_{(t;\st{1}{1})}(x)$ (for $t$ not equal to the single vertex tree).

First notice that $g_{(t; \st{1}{1})}(x) = x + g_{(t;\st{3}{1})}(x)$. This is because, unless $t$ is the single vertex tree, the generating function for trees with a single vertex at the root will always account for the one tree with one leaf.  Then, the rest of the trees avoiding $t$ have a \st{3}{1} pattern at the root, so $g_t(x)=g_{(t; \st{1}{1})}(x) = x + g_{(t;\st{3}{1})}(x)$.

Next, we have introduced a new generating function $g_{(t;\st{3}{1})}(x)$, and we need to derive a new recurrence for this function. To do this, we recognize that a $t$-avoiding tree with pattern $p$ at its root is made up of three $t$-avoiding subtrees with $p_\ell$, $p_c$, and $p_r$ at their respective roots. However, $g_{(t;p_\ell )}(x)$, $g_{(t;p_c)}(x)$, and $g_{(t;p_r)}(x)$ each account for trees avoiding $t$ individually, which overcounts when the trees have root pattern of $p_\ell \cap t_\ell$, $p_c \cap t_c$, and $p_r \cap t_r$ respectively. Therefore, we have $$g_{(t;p)}(x)= g_{(t;p_\ell)}(x) \cdot g_{(t;p_c)}(x) \cdot g_{(t;p_r)}(x) - g_{(t;p_\ell \cap t_\ell)}(x) \cdot g_{(t; p_c \cap t_c)}(x) \cdot g_{(t; p_r \cap t_r)}(x).$$  This observation holds not only for the tree pattern \st{3}{1}, but for any non-trivial tree pattern $p$.  We repeatedly use this observation to derive a new recurrence for each generating function $g_{(t;p)}(x)$ that arises in our computation until we have a complete system of equations.  We are guaranteed that such a system will eventually be complete since if $t$ has depth $d$, each pattern $p$ introduced in this process has depth at most $d$ and there are finitely many tree patterns of depth at most $d$.  Once we have a complete system of equations, we eliminate all unwanted variables until we have a functional equation for $g_{(t;\st{1}{1})}(x)$. Here, then, is the algorithm to compute $g_t(x)$ for any non-trivial ternary tree pattern $t$.

\begin{enumerate}
\item Initialize $\text{Eq}=\{g_{(t;\st{1}{1})}(x) = x + g_{(t;\st{3}{1})}(x)\}$, $Var=\{\st{1}{1}\}$, $\text{P}=\{\st{3}{1}\}$, and $P1=\emptyset$.
\item For $p \in \text{P}$, do:
\begin{itemize}
\item Let $\text{Eq} = \text{Eq} \cup \{g_{(t;p)}(x) = g_{(t;p_\ell)}(x) \cdot g_{(t;p_c)}(x) \cdot g_{(t;p_r)}(x) - g_{(t;p_\ell \cap t_\ell)}(x) \cdot g_{(t;p_c \cap t_c)}(x) \cdot g_{(t;p_r \cap t_r)}(x)\}$.
\item Let $\text{Var}=\text{Var} \cup \{p\}$.
\item Let $\text{P1}=(\text{P1} \cup \{p_\ell,p_c,p_r,p_\ell \cap t_\ell, p_c \cap t_c, p_r \cap t_r\}) \setminus \text{Var}$.
\end{itemize}
\item If $\text{P1} \neq \emptyset$ then let $\text{P}=\text{P1}$, $\text{P1}=\emptyset$, and go to Step 2.  

If $\text{P1}=\emptyset$ then eliminate all variables in $\text{Var} \setminus \{g_{(t;\st{1}{1})}(x)\}$ from the system of equations $\text{Eq}$ to compute a functional equation for $g_{(t;\st{1}{1})}(x)$.
\end{enumerate}

We illustrate this algorithm by using it to compute the avoidance generating function for $t_{7_3}$.

In step 1, we initialize
 
\begin{align*}
\text{Eq}&=\{g_{(t_{7_3};\st{1}{1})}(x)= x+g_{(t_{7_3};\st{3}{1})}(x)\}\\
\text{Var}&=\{\st{1}{1}\}\\
\text{P}&=\{\st{3}{1}\}\\
\text{P1}&=\emptyset\\
\end{align*}

In step 2, we consider $\st{3}{1} \in \text{P}$ to obtain
\begin{align*}
\text{Eq}&=\text{Eq} \cup \{g_{(t_{7_3};\st{3}{1})}(x)=(g_{(t_{7_3};\st{1}{1})}(x))^3-g_{(t_{7_3};\st{5}{1})}(x) \cdot (g_{(t_{7_3};\st{1}{1})}(x))^2\}\\
\text{Var}&= \text{Var} \cup \{\st{3}{1}\}\\
\text{P1}&=\{\st{5}{1}\}\\
\end{align*}

Since $\text{P1} \neq \emptyset$, relabel $\text{P1}$ as $\text{P}$, and consider $\st{5}{1} \in \text{P}$ to obtain

\begin{align*}
\text{Eq}&= \text{Eq} \cup \{g_{(t_{7_3};\st{5}{1})}(x)=g_{(t_{7_3};\st{3}{1})}(x)\cdot (g_{(t_{7_3};\st{1}{1})}(x))^2-g_{(t_{7_3};\st{5}{1})}(x)\cdot (g_{(t_{7_3};\st{1}{1})}(x))^2\}\\
\text{Var}&=\text{Var} \cup \{\st{5}{1}\}\\
\text{P1}&=\emptyset\\
\end{align*}

Since $\text{P1}=\emptyset$, we consider the three equations in $\text{Eq}$.  Let $a=g_{(t_{7_3};\st{1}{1})}(x)$, $b=g_{(t_{7_3};\st{3}{1})}(x)$, and $c=g_{(t_{7_3};\st{5}{1})}(x)$.  Then, 

\begin{align*}
a&=x+b\\
b&=a^3-ca^2\\
c&=ba^2-ca^2
\end{align*}

Eliminating $b$ and $c$ gives the equation $xa^4+xa^2-a+x=0$. 

For very simple trees, we can usually solve the resulting functional equation directly for $g_{(t;\st{1}{1})}(x)$; however, this quartic functional equation is a more characteristic result for complex tree patterns. Although this functional equation does not have a simple explicit solution, we can use it to compute arbitrarily many coefficients of the generating function $g_{t_{7_3}}(x)$ by making the substitution  $a=\sum_{n=0}^k \text{av}(n)x^n$, isolating the coefficients of each power of $x$, and setting them each equal to zero.  From the functional equation $xa^4+xa^2-a+x=0$, we find the sequence $\text{av}(n)$, $0 \leq n \leq 25$, to be 0, 1, 0, 1, 0, 3, 0, 11, 0, 46, 0, 207, 0, 979, 0, 4797, 0, 24138, 0, 123998, 0, 647615, 0, 3428493, 0, 18356714,$\dots$ (OEIS \seqnum{A006605} interspersed by zeros).

A complete classification of ternary trees $t$ with up to 9 leaves is given in the Appendix, along with functional equations for $g_t(x)$ and 20 terms of the corresponding avoidance sequences.

Note that given a tree pattern $t$, the algorithm given in this section generates a system of equations each of which has maximum total degree 3.  Auxiliary variables in this system will be eliminated to produce a polynomial functional equation for $g_t(x)$.  This guarantees that $g_t(x)$ is always algebraic.  Khoroshkin and Piontkovski \cite{KPTBA} independently showed that generating functions are algebraic in the case of general pattern-avoiding trees; however, their work is done in a different context.

\section{Bijections on Ternary Trees}\label{S:bijections}
Now that we have discussed two methods for enumerating pattern-avoiding trees, we look for connections between specific sets of those trees.  Recall that several of the ternary trees in this paper have had the same pattern avoidance sequence as one or more other trees.  That is, for some distinct $i$ and $j$, we found that $g_{t_{k_i}}(x) = g_{t_{k_j}}(x)$.  Such patterns are said to be \emph{Wilf equivalent}.  We will now explain why certain pairs of tree patterns $\{t_{k_i},t_{k_j}\}$ are Wilf equivalent. As Rowland did, we accomplish this through finding bijections between the members of $\text{Av}_{t_{k_i}}(n)$ and those of $\text{Av}_{t_{k_j}}(n)$.  In order to present these bijections in a clear and concise way, we first present an alternate notation for ternary trees.  We use this notation to describe bijections that explain all Wilf equivalences between 5 and 7-leaf ternary tree patterns.  We then generalize these maps.

\subsection{Word Notation for Trees}
In this subsection we represent trees as sets of integer words.  This notation easily extends to $m$-ary trees (i.e., trees where each vertex has 0 or $m$ children). At the foundation of this word representation are $m$-leaf parents.

\begin{definition} An \emph{$m$-leaf parent} is an internal vertex, $v$, of an $m$-ary tree such that $v$ has exactly $m$ children, all of which are leaves. \end{definition}

For example, $t_{7_4}$ has one 3-leaf parent and $t_{7_1}$ has two 3-leaf parents.  

Word notation represents an $m$-ary tree with a set of words, where each word follows the path from the root to an $m$-leaf parent.  We construct such a set from the following Tree-Set Algorithm:

\begin{enumerate}
\item Label the children of each internal vertex of an $m$-ary tree from left to right, $1$ through $m$.  (In ternary trees, then, a vertex's left child is labeled $1$, its center child $2$, and its right child $3$.)
\item Denote a path from the root of a tree to an $m$-leaf parent by an integer word $x_1 \cdots x_k$, where $k$ is the length of the path from the root to the $m$-leaf parent, such that $x_i \in \mathbb{Z}$ and $1\leq x_i \leq m$ for $1 \leq i \leq k$. The first number $x_1$ in the word represents the child of the root labeled $x_1$; $x_i$ then refers to the child of the vertex given by $x_{i-1}$ that is labeled with the value of $x_i$. 
\end{enumerate}

As an example, let us look again look at $t_{7_4}$ and $t_{7_1}$.  In $t_{7_4}$, the only $3$-leaf parent is reached by a path beginning at the root, going to the root's left child, then to this vertex's center child; in word notation, $t_{7_4}$ is denoted by the set $\{ 12 \}$.  For $t_{7_1}$, we reach one of its 3-leaf parents by going to the root's left child, and the other by the root's center child; this becomes a set $\{ 1,2 \}$.  Note that an ordered $m$-ary tree $T$ is uniquely defined by the set of paths from its root to each $m$-leaf parent.  The single vertex tree is represented by the empty set, $\{\}$, and the 3-leaf tree ($t_{3_1}$) by the set containing the empty word, $\{\epsilon\}$.  The set $\{21, 23, 321\}$ denotes the tree $T$ given in Figure \ref{Feg1} and the set $\{13, 223\}$ denotes the ternary tree given in Figure \ref{Fi:schroder}.  Also, note that the Tree-Set Algorithm above is clearly reversible.  In particular:

\begin{enumerate}
\item Create an $m$-ary tree from each word by the following procedure:
\begin{enumerate}
\item Create a root.
\item Give the root $m$ children, labeled left to right from $1$ to $m$.
\item For the word $x_1x_2 \cdots x_k$ give the $x_1$-st child of the root $m$ children. Label these children 1 through $m$ as before, repeating the process at each level where $x_i$ denotes giving $m$ children to the $x_i$-th child of the vertex that was given children by $x_{i-1}$. 
\end{enumerate}
\item Take the intersection of all trees obtained from step 1 to find the final $m$-ary tree.
\end{enumerate}

We note that representing trees as words or as sets of words is not a new idea.  For example, Stanley \cite{Stanley99} shows how to represent plane trees as certain integer sequences and as parenthesizations of words subject to certain constraints. 

Now that we have this alternate representation of $m$-ary trees, we consider specific properties of the word notation that are relevant to our questions of pattern avoidance.  First note that any set of words where one word is a prefix of another is redundant for the representation of trees.  Since each word specifies a path within a tree, if word $w_1$ is a prefix of word $w_2$, the path specified by $w_1$ is necessarily a part of the path specified by $w_2$, and including $w_1$ is superfluous.

\begin{definition}
Let \emph{$S$} be a set of words on the alphabet $\{1, \dots, m\}$, such that no word is a prefix of another word in $S$. Namely, $S$ is an arbitrary set of words describing an $m$-ary tree.  We write \emph{$A(m)$} for the set of all such sets $S$.
\end{definition}

Consider a set of words $ \{ L_i \}_{i=1}^l $ in $A(m)$ corresponding to tree pattern $t$.  Note that tree $T$, given by $ \{ M_h \}_{h=1}^p $, contains $t$ if there exist $ \{ M_{h_i} \}_{i=1}^l $ where each $M_{h_i}$ begins with the same (possibly empty) prefix as all other $M_{h_i}$'s, followed by exactly the ordered sequence of elements of $L_i$; this may or may not be then followed by an additional sequence of numbers.

For example, the tree pattern $t= \{1323, 1223\} $, is contained by $$T_1= \{ 323\mathbf{1323}, 11322, 323\mathbf{1223}112 \}. $$  Notice that, after the prefix $323$, the first and third words of $T_1$ have exactly the sequence of each word of $t$.  However, $T_2= \{ 31323, 1223 \} $ avoids $t$.  Even though it contains each sequence of numbers from the words of $t$, $T_2$'s words do not begin with the same prefix before the sequences begin.

In the following subsections, $B_{t_{k_i},t_{k_j}}:A(m) \rightarrow A(m)$ will denote a bijection from trees avoiding $t_{k_i}$ to trees avoiding $t_{k_j}$.  We find such maps by analyzing the word notation for our pattern-avoiding trees.  We are now ready to present bijections between Wilf equivalent tree patterns.

\subsection{The patterns \texorpdfstring{$t_{5_1}$}{t51} and \texorpdfstring{$t_{5_2}$}{t52}}

A tree avoids $t_{5_1}$ if none of its left vertices have children, and avoids $t_{5_2}$ if none of its center vertices have children.  In order to map a tree $T$ avoiding $t_{5_1}$ to a tree avoiding $t_{5_2}$, we define our bijection $B_{t_{5_1},t_{5_2}}(T)$ to ``switch'' the center subtree of every vertex with the left subtree of the same vertex.  In terms of word notation, a tree avoids $t_{5_1}$ if it has no 1's in its words, and avoids $t_{5_2}$ if it has no 2's in its words.  Thus, we define $B_{t_{5_1},t_{5_2}}(T)$ to replace every 1 in $T$'s words with 2, and every 2 with a 1.  For example, $B_{t_{5_1},t_{5_2}}( \{ 233, 32 \} ) = \{ 133, 31 \} $.  It is clear that this map is one-to-one, onto, and preserves number of leaves.

\subsection{The patterns \texorpdfstring{$t_{7_3}$}{t73} and \texorpdfstring{$t_{7_7}$}{t77}}

A tree avoids $t_{7_3}$ if no two consecutive left vertices have children, and avoids $t_{7_7}$ if no two consecutive center vertices have children.  In word notation, then, a tree avoids $t_{7_3}$ if it has no pairs of consecutive 1's, and it avoids $t_{7_7}$ if it has no pairs of consecutive 2's.  Thus, we define $B_{t_{7_3},t_{7_7}}(T) = B_{t_{5_1},t_{5_2}}(T)$.  That is, for a tree $T$ avoiding $t_{7_3}$, $B_{t_{7_3},t_{7_7}}(T)$ replaces each 1 in $T$'s set of words with a 2, and each 2 with a 1.  Because we originally defined $B_{t_{5_1},t_{5_2}}(T)$ on $A(3)$ defining $B_{t_{7_3},t_{7_7}}(T)$ in this way is reasonable.  For example, $B_{t_{7_3},t_{7_7}}(\{13, 22, 322\}) = \{23, 11, 311\}$.  As stated before, this map is a bijection.

\subsection{The patterns \texorpdfstring{$t_{7_1}$}{t71} and \texorpdfstring{$t_{7_2}$}{t72}}

For a tree $T$ to avoid $t_{7_1}$, no vertex $v$ can have children descending from both its left and center children; to avoid $t_{7_2}$, no vertex can have children descending from both its left and right children.  Therefore, we define a bijection $B_{t_{7_1},t_{7_2}}(T)$ to switch the right and center subtrees of each vertex.  Using word notation, this is equivalent to defining $B_{t_{7_1},t_{7_2}}(T)$ to replace every 2 with a 3 and every 3 with a 2.  For example, $B_{t_{7_1},t_{7_2}}( \{ 121, 1232, 322, 331 \} ) = \{ 131, 1323, 221, 233\}$.  Again, it is clear that $B_{t_{7_1},t_{7_2}}$ is one-to-one, onto, and preserves number of leaves.

\subsection{The patterns \texorpdfstring{$t_{7_4}$}{t74} and \texorpdfstring{$t_{7_5}$}{t75}}

Using word notation, a tree avoids $t_{7_4}$ if none of its words have a 1 followed by a 2; it avoids $t_{7_5}$ if none of its words have a 1 followed by a 3.  Therefore, we define $B_{t_{7_4},t_{7_5}}(T)=B_{t_{7_1},t_{7_2}}(T)$.  That is, $B_{t_{7_4},t_{7_5}}(T)$ replaces each 2 with a 3, and each 3 with a 2. For example, $B_{t_{7_4},t_{7_5}}(\{1313, 3213, 323\}) = \{1212, 2312, 232\}$.  As in all previous examples, it is clear that $B_{t_{7_4}, t_{7_5}}$ is a bijection that preserves number of leaves.

\subsection{The patterns \texorpdfstring{$t_{7_5}$}{t75} and \texorpdfstring{$t_{7_6}$}{t76}}

Using word notation, a tree avoids $t_{7_5}$ if none of its words have a 1 followed by a 3; it avoids $t_{7_6}$ if none of its words have a 2 followed by a 1.  Therefore, we define $B_{t_{7_5},t_{7_6}}(T)$  to replace every 1 with a 2, every 2 with a 3, and every 3 with a 1. $B_{t_{7_5},t_{7_6}}$ clearly maps words that do not contain the sequence 13 to words that do not contain the sequence 21.  For example, $B_{t_{7_5},t_{7_6}}(\{1, 21, 3212\}) = \{2, 32, 1323\}$.  As in all previous examples, it is again clear that  $B_{t_{7_5},t_{7_6}}$ is a bijection that preserves number of leaves.

\subsection{The patterns \texorpdfstring{$t_{7_1}$}{t71} and \texorpdfstring{$t_{7_4}$}{t74}}

This bijection is a bit more complicated than the previous maps.  Consider the function $B_{t_{7_1},t_{7_4}}(T)$ defined by the following procedure applied each word $L \in T$:
\begin{enumerate}
\item If $L$ contains no 1 followed by a 2, do nothing to $L$.
\item Otherwise, locate the first copy of $1 \cdots 12$ in $L=x_1 \cdots x_{\left|L\right|}$.  In particular, suppose that $x_1 \cdots x_{i-1}$ has no copy of $1 \cdots 12$, $x_{i-1}\neq 1$, $x_i=x_{i+1}=\cdots=x_j=1$ and $x_{j+1}=2$. We map $L$ to the pair of words $x_1 \cdots x_j,x_1 \cdots x_{i-1}x_{j+1} \cdots x_{\left|L\right|}$. So $x_1 \cdots x_j$ contains no occurrence of $12$.
\item Iterate step 2 for each new word created until we have produced $L_1,L_2,\dots,L_p$, none of which contain a 1 followed by a 2.
\end{enumerate}

For example, if $T_1=\{1232311121\}$, our first iteration maps $T_1$ to $\{1,232311121\}$.  Our second iteration gives $\{1,2323111,232321\}$.  All of these words are $1 \cdots 12$-free, so we have $$B_{t_{7_1},t_{7_4}}(T_1)= \{1,2323111,232321\}.$$ 

In order to prove that $B_{t_{7_1},t_{7_4}}$ is a bijection, we first construct an inverse function, $B_{t_{7_1},t_{7_4}}^{-1}(T)$.  We will start at the root of tree $T$ and work downward to consider all occurrences of $t_{7_1}$; that is, word pairs of the form $p_01s_1,p_02s_2$, where $p_0$ denotes a common prefix in the two words and $s_1,s_2$ are suffixes. Note that for each occurrence of $t_{7_1}$ it is possible for there to be multiple words of the form $p_02s_i$. Our inverse map is given by the following process:

\begin{enumerate}
\item On each level of $T$ beginning at the root, replace each occurrence of $p_02s_i$ that is part of an occurrence of $t_{7_1}$ with $p_012s_i$. If there are multiple occurrences of $t_{7_1}$ on the same level, then applying step 2 to one occurrence of $t_{7_1}$ does not affect the words denoting any other occurrence of $t_{7_1}$ on that level. Therefore, the order with which we apply this step to each occurrence of $t_{7_1}$ at the $i$th level is irrelevant.
\item Iterate step 1 at each successive level, beginning with the root.
\item Discard any words that are a prefix of another in $T$'s set of words.
\end{enumerate}

For example, 
\begin{itemize}
\item For $T_2 = \{1,2323111,232321 \}$, at the first level we have an occurrence of $t_3$ given by $1 \cap 2323111$ and from $1 \cap 232321$. So, from step 1, we replace $\{2323111, 232321\}$ with \\$\{12323111, 1232321\}$.
\item We now have the set $\{1,12323111,1232321\}$; step 2 requires that we check the next levels in order from lowest to highest and we see that $t_3$ does not occur until the sixth level, and is given by the words $12323\textbf{1}11,12323\textbf{2}1$.  Thus, we replace $1232321$ with $12323\textbf{1}21$ to obtain $\{1,12323111,12323121\}$.
\item With the third iteration of step 1, we replace $12323121$ with $12323\textbf{1}121$.  The fourth iteration replaces $123231121$ with $12323\textbf{1}1121$, and we are left with the set $\{1, \allowbreak 12323111, \allowbreak 1232311121\}$. After applying step 3, we see that $B_{t_{7_1},t_{7_4}}^{-1}(T_2)=\{1232311121\}$. We note that $B_{t_{7_1},t_{7_4}}^{-1}(T_2)= T_1$ as expected from our earlier example. 
\end{itemize}

It is easy to see that $B_{t_{7_1}t_{7_4}}(T)$ does in fact map trees avoiding $t_{7_1}$ to trees avoiding $t_{7_4}$.  It is also clear that $B_{t_{7_1},t_{7_4}}^{-1}(B_{t_{7_1},t_{7_4}}(T))=T$, so $B_{t_{7_1},t_{7_4}}(T)$ has a clear inverse.  Thus $B_{t_{7_1},t_{7_4}}(T)$ is a bijection.

We further claim that $B_{t_{7_1},t_{7_4}}(T)$ preserves the numbers of leaves of $T$.  Step 1 in our bijection is the only step that changes the structure of $T$. Consider an arbitrary occurrence of $t_{7_4}$ such that the path from the root of $T$ to the root of the occurrence is given by the prefix $p_0$ and such that there is no occurrence of $t_{7_4}$ in $p_0$. Then step 1 in our bijection will map all words with the prefix $p_012$ to the word $p_01$ and words with the prefix $p_02$. As a result of no longer having any words with the prefix $p_012$, we see that the vertex given by the path $p_01$ has one more leaf in $B_{t_{7_1},t_{7_4}}(T)$ than it did in $T$ (namely its second child is a leaf, but was not a leaf in $T$). However we also see that the vertex given by the path $p_0$ has one less child that is a leaf as a result of having words with $p_02$ as a prefix. There could not have already been a word in $T$ with $p_02$ as a prefix since this would entail having an occurrence of $t_{7_1}$ at the vertex given by the path $p_0$. Having the word $p_01$ does not add to or subtract from the number of leaves we have since it is a prefix of the word it replaces and thus creates no new vertices. With each iteration of step 2 on an occurrence of $t_{7_4}$ this same reasoning holds. So we see that the number of leaves in $B_{t_{7_1},t_{7_4}}(T)$ is the same as the number of leaves in $T$.

We have now accounted for all Wilf equivalences between tree patterns with 5 or 7 leaves.

\subsection{General Approaches to Bijections}

In this section, we generalize the previous six bijections to bijections for a larger class of tree patterns.  In particular, we present a bijection between certain tree patterns that have the same avoidance generating function and have only one $m$-leaf parent.  These are tree patterns that are represented by a single word in word notation.

Consider two $m$-ary tree patterns $\{L_1\}$ and $\{L_2\}$ in word notation.  If there exists a bijection $b: \{1,\dots , m\} \to \{1, \dots , m\}$ that maps $L_1$ to $L_2$, then there is a bijection $B_{L_1,L_2}(T)$ from the set of trees avoiding $\{L_1\}$ to the set of trees avoiding $\{L_2\}$.  In particular, $B_{L_1,L_2}(T)$ is the map that replaces each letter $i$ in the lists of an $L_1$-avoiding tree with $b(i)$.

For example, if $\{L_1\} = \{11213\}$, $\{L_2\} = \{22321\}$, the map $b$ sends $1 \mapsto 2$, $2 \mapsto 3$, and $3 \mapsto 1$.  If we consider a tree $T$ that avoids $\{L_1\}$ given in list notation, then $B_{L_1,L_2}(T)$ replaces each integer $i$ with $b(i)$.  For example $B_{L_1,L_2}(\{13231,22321\}) = \{21312,33132\}$.

In general, if $T$ is a tree denoted by a set of words that avoids $\{L_1\}$, when we apply the bijection $b$ to the letters of $T$, we have $B_{L_1,L_2}(T)$ to be a set of words in which no word contains any instance of $\{L_2\}$, that is, $B_{L_1,L_2}(T)$ avoids $\{L_2\}$.

\begin{theorem} 
\label{superpattern}
The map $B_{L_1,L_2}(T)$ is one-to-one, onto, and preserves the number of leaves of $T$.
\end{theorem}

\begin{proof}
The fact that $B_{L_1,L_2}(T)$ is one-to-one and onto follows from the fact that $B_{L_1,L_2}(T)$ is induced by the map $b$, which is just a permutation of $\{1,\dots,m\}$.

To show that $B_{L_1,L_2}(T)$ preserves the number of leaves of $T$, it is enough to show that for each internal vertex $v_1$ in a given tree $T$ there is a unique internal vertex $v_2$ in $B_{L_1,L_2}(T)$ that has the same number of children that are leaves as does $v_1$. Consider the word $p_1$ describing the path to $v_1$ in $T$. Let $x$ be the number of distinct letters in $\{1,\dots,m\}$ that follow $p$ as a prefix in a word of $T$.  Then $v_1$ has $m-x$ children that are leaves. Now consider $B_{L_1,L_2}(\{p_1\})=\{p_2\}$.  Since $x$ distinct letters followed $p_1$ in $T$, there must be $x$ distinct letters that follow $p_2$ as a prefix in $B_{L_1,L_2}(T)$, meaning the vertex whose path from the root is given by $p_2$ has $m-x$ children that are leaves.  This holds true for every vertex of $T$, proving that the number of leaves remains the same.
\end{proof}

These generalized bijections for trees have one more interesting property.  Consider the word notation for an $m$-ary tree.  This can also be read as the word notation for an $M$-ary tree for any $M \geq m$.  This says that once we have a bijection $B_{L_1,L_2}(T)$ for a pair of $m$-ary trees, we have necessarily discovered a Wilf equivalence for corresponding pairs of $M$-ary trees for each $M \geq m$.

\subsection{Final notes on bijections}

In this section we have given bijections explaining all non-trivial Wilf equivalences between ternary tree patterns with 5 or 7 leaves.  We note that all but one of the bijections given in Section \ref{S:bijections} are guaranteed to exist by the results of Section 4.8.  The bijection between $t_{7_1}$- and $t_{7_4}$-avoiding trees is the lone exception to this method, since it involved ``cutting'' the integer words representing trees, rather than just relabeling them. In the appendix, we also give computational data for equivalent ternary tree patterns with 9 leaves; it turns out that all 9-leaf ternary tree patterns fall into just three distinct Wilf classes.  Many, but not all, of these equivalences can also be explained with the generalized bijections of Section 4.8.

Moreover, all bijections except the one for $t_{7_1}$ and $t_{7_4}$ may clearly be seen as replacement bijections in the sense of Rowland's binary tree patterns paper \cite{Rowland10}.  In fact, we can see that this last bijection is not a replacement bijection by considering a well-chosen example.  Consider the tree whose word representation is $\{1232311121\}$, as shown in Figure \ref{nonreplace}. This tree contains two copies of the tree pattern $t_{7_4}$, or $\{12\}$.  One of those copies is at the root, and the other is near the bottom of the tree, with one copy of $t_{3_1}$ appended to its deepest left leaf.  On the other hand, the image of this tree is represented by $\{1,2323111,232321\}$ (also shown in Figure \ref{nonreplace}).  This tree contains two copies of the tree $t_{7_1}$, or $\{1,2\}$, one at the root, and one deeper in the tree.  However, this lower copy has a copy of $t_{5_1}$ and a copy of $t_{3_1}$ appended to its deepest leaves.  In a true replacement bijection, the copies of $t_{7_4}$ should be transformed to copies of $t_{7_1}$, \emph{and} the subtrees descending from leaves of a copy of $t_{7_4}$ should be moved in entirety to be subtrees descending from leaves of a copy of $t_{7_1}$.  This is clearly not the case in this bijection.  It remains an open question to determine whether there exists a replacement bijection between ternary trees avoiding $t_{7_1}$ and ternary trees avoiding $t_{7_4}$.

\begin{figure}[bht]
\begin{center}
\includegraphics{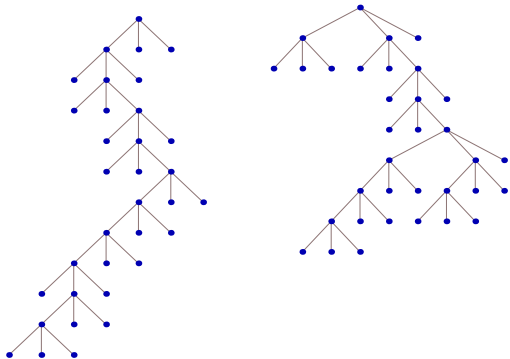}
\end{center}
\caption{The ternary trees with word representations $\{1232311121\}$ and $\{1, \allowbreak 2323111, \allowbreak 232321\}$}
\label{nonreplace}
\end{figure}

Many of our bijections overlap with Rowland's idea of replacement bijections; however, we propose that considering trees as sets of integer words provides more insight into the process of developing Wilf equivalent tree patterns.

\section{Conclusion}

Throughout this paper, we have investigated pattern avoidance in ternary trees, extending previous work for binary trees.  We began by finding recurrence relations and generating functions by hand for several simple ternary tree patterns.  To make the computation of avoidance sequences easier, however, we developed an algorithm, based on Rowland's algorithm for binary trees, to find the generating function for trees avoiding any given tree pattern.  Next we classified the tree patterns, grouping together those with the same avoidance sequence. From here, we were able to find bijections between the sets of trees avoiding specific pairs of two equivalent tree patterns, $t_{k_i}$ and $t_{k_j}$; for these pairs of trees, we transformed any tree avoiding $t_{k_i}$ into one that avoids $t_{k_j}$.  After stating several bijections between specific pairs of tree patterns, we then generalized this to bijections between trees avoiding patterns in the same equivalence class of trees under permutations of $\{1, \dots, m\}$ .

\section{Acknowledgement}
The authors thank Eric Rowland for a number of presentation suggestions and for support with generating the many tree graphics required for this paper.

\section*{Appendix}

This appendix lists ternary trees with at most nine leaves, classifying them by their avoidance generating function and sequence.  For each class, we give a functional equation satisfied by $a=g_t(x)$, and we list the first 20 terms (including zeros) of the corresponding avoidance sequence.  If $g_t(x)$ has a simple explicit form it is included, and if the avoidance sequence for a class is listed in the Online Encyclopedia of Integer Sequences (without interspersed zeros) \cite{OEIS}, we include the appropriate reference.  For brevity, left--right reflections are omitted.

\vspace{5mm}

\hrule

\vspace{5mm}

\textit{Class 5}

\begin{center}
\includegraphics{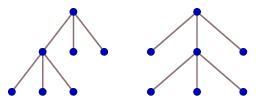}
\end{center}

\begin{itemize}
\item $xa^2-a+x=0$
\item $g_t(x)=\frac{1- \sqrt{1-4x^2}}{2x}$
\item $0, 1, 0, 1, 0, 2, 0, 5, 0, 14, 0, 42, 0, 132, 0, 429, 0, 1430, 0, 4862,\dots$
\item OEIS \seqnum{A000108}: Catalan numbers
\end{itemize}

\hrule 

\vspace{5mm}

\textit{Class 7.1}

\begin{center}
\includegraphics{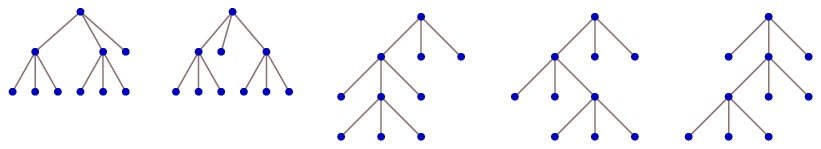}
\end{center}

\begin{itemize}
\item $2xa^2-x^2a-a+x=0$
\item $g_t(x)=\frac{(x^2+1)- \sqrt{(x^2+1)^2-8x^2}}{4x}$
\item $0, 1, 0, 1, 0, 3, 0, 11, 0, 45, 0, 197, 0, 903, 0, 4279, 0, 20793, 0, 103049,\dots$
\item OEIS \seqnum{A001003}: Little Schr\"{o}der numbers
\end{itemize}

\hrule

\vspace{5mm}

\textit{Class 7.2}

\begin{center}
\includegraphics{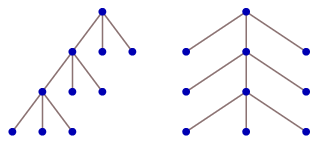}
\end{center}

\begin{itemize}
\item $xa^4+xa^2-a+x=0$
\item $0, 1, 0, 1, 0, 3, 0, 11, 0, 46, 0, 207, 0, 979, 0, 4797, 0, 24138, 0, 123998,\dots$
\item OEIS \seqnum{A006605}: number of modes of connections of $2n$ points
\end{itemize}

\hrule
 
\vspace{5mm}

\textit{Class 9.1}

\begin{center}
\includegraphics{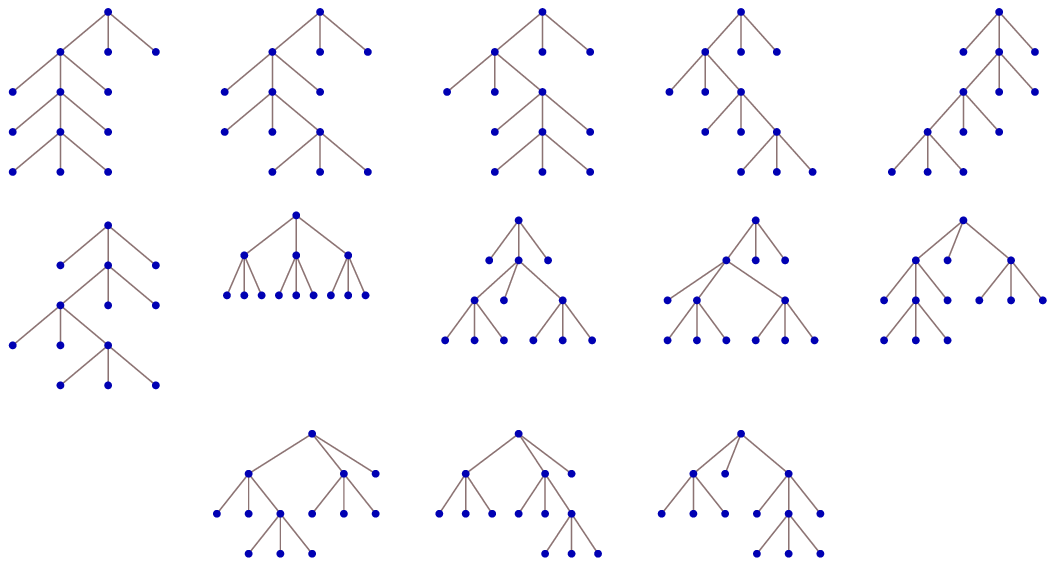}
\end{center}

\begin{itemize}
\item $ 3xa^2-3x^2a-a+x^3+x=0$
\item $g_t(x)=\frac{(3x^2+1)- \sqrt{1-6x^2-3x^4}}{6x}$
\item $0, 1, 0, 1, 0, 3, 0, 12, 0, 54, 0, 261, 0, 1323, 0, 6939, 0, 37341, 0, 205011, \dots$
\item OEIS A107264:	Expansion of $\frac{-(3x-1)-\sqrt{1-6x-3x^2}}{6 x^2}$. 
\end{itemize}

\hrule 
\vspace{5mm}

\textit{Class 9.2}

\begin{center}
\includegraphics{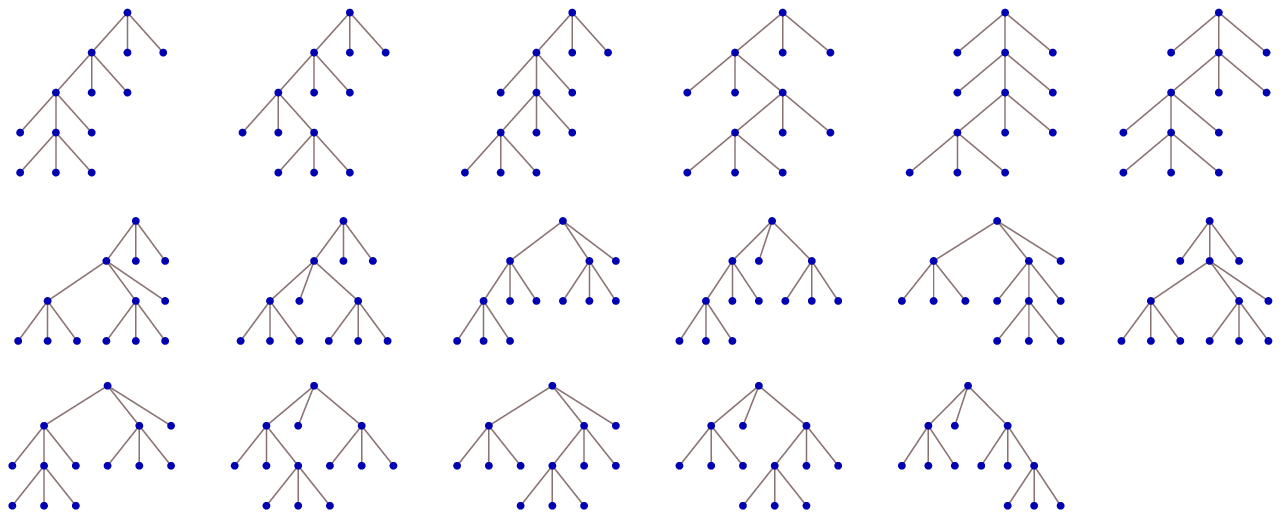}
\end{center}

\begin{itemize}
\item $xa^4-x^2a^3+2xa^2-x^2a-a+x=0$
\item $0, 1, 0, 1, 0, 3, 0, 12, 0, 54, 0, 261, 0, 1324, 0, 6954, 0, 37493, 0, 206316, \dots$
\item OEIS A200740: Generating function satisfies $x^2A^4(x)-x^2A^3(x)+2xA^2(x)-xA(x)-A(x)+1=0$
\end{itemize}

\hrule 

\vspace{5mm}

\textit{Class 9.3}

\begin{center}
\includegraphics{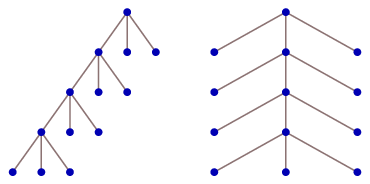}
\end{center}

\begin{itemize}
\item $xa^6+xa^4+xa^2-a+x=0$
\item $0, 1, 0, 1, 0, 3, 0, 12, 0, 54, 0, 262, 0, 1337, 0, 7072, 0, 38426, 0, 213197, \dots$
\item OEIS A186241: Generating function given by $x^3A^6(x)+x^2A^4(x)+xA^2(x)-A(x)+1=0$
\end{itemize}

\hrule


\begin{thebibliography}{1}

\bibitem{DTBA} V. Dotsenko, Pattern avoidance in labelled trees, preprint, \url{http://arxiv.org/abs/1110.0844}. 

\bibitem{FS09} P. Flajolet and R. Sedgewick, \emph{Analytic Combinatorics}, Cambridge University Press, 2009.

\bibitem{FSS90} P. Flajolet, P. Sipala, and J. M. Steyaert, Analytic
variations on the common subexpression problem, {\it Automata,
Languages, and Programming:  Proc. of ICALP 1990},
Lecture Notes in Computer Science, Vol.\ 443,
Springer, 1990, pp.\ 220--234.

\bibitem{GJ79} I. Goulden and D. Jackson, An inversion theorem for
cluster decompositions of sequences with distinguished subsequences,
\emph{J. London Math. Soc.} 
\textbf{20} (1979), 567--576.

\bibitem{KPTBA} A. Khoroshkin and D. Piontkovski, On generating series
of finitely presented operads, in preparation.

\bibitem{NZ99} J. Noonan and D. Zeilberger, The Goulden-Jackson cluster
method: extensions, applications, and implementations, \emph{J.
Difference Equations and Applications} \textbf{5} (1999), 355--377.

\bibitem{Rowland10} E. S. Rowland, Pattern avoidance in binary trees,
\emph{J. Combin. Theory, Ser.\ A} \textbf{117} (2010),
741--758.

\bibitem{Stanley99} R. P. Stanley, \emph{Enumerative Combinatorics},
Cambridge University Press, 1999.

\bibitem{OEIS} N. Sloane, The Encyclopedia of Integer Sequences.
Available at \url{http://oeis.org}, 2011.

\bibitem{SF83} J. M. Steyaert and P. Flajolet, Patterns and
pattern-matching in trees: an analysis, \emph{Info. Control}
\textbf{58} (1983), 19--58.

\end{thebibliography}
\end{document}